\documentclass[notitlepage,reqno,11pt]{amsart}
\RequirePackage[OT1]{fontenc}
\usepackage[foot]{amsaddr}
\RequirePackage{amssymb,nicefrac,bm,upgreek,mathtools,verbatim,enumerate}
\RequirePackage[mathscr]{eucal}
\RequirePackage{dsfont}
\RequirePackage[normalem]{ulem}

\usepackage[final]{hyperref}

\usepackage[margin=1in]{geometry}
\numberwithin{equation}{section}

\theoremstyle{plain}
\newtheorem{theorem}{Theorem}[section]
\newtheorem{lemma}{Lemma}[section]

\theoremstyle{definition}

\newtheorem{definition}{Definition}[section]

\theoremstyle{remark}
\newtheorem{remark}{Remark}[section]

\newcommand{\stkout}[1]{\ifmmode\text{\sout{\ensuremath{#1}}}\else\sout{#1}\fi}
\usepackage{color}
\definecolor{mred}{rgb}{.7,.0,.0}
\definecolor{dred}{rgb}{.6,.0,.0}
\definecolor{mblue}{rgb}{.0,.0,.8}
\definecolor{dblue}{rgb}{.0,.0,.5}
\definecolor{mgreen}{rgb}{.0,.6,.0}
\definecolor{dgreen}{rgb}{.0,.4,.0}
\definecolor{dmagenta}{rgb}{.4,.1,.5}

\hypersetup{
 colorlinks=true,
 citecolor=mblue,
 linkcolor=mblue,
 frenchlinks=false,
 pdfborder={0 0 0},
 naturalnames=false,
 hypertexnames=false,
 breaklinks}
\usepackage[capitalize,nameinlink]{cleveref}
\usepackage[msc-links,abbrev,nobysame]{amsrefs}
\crefname{section}{Section}{Sections}
\crefname{subsection}{subsection}{subsections}
\crefname{notation}{Notation}{Notations}
\crefname{hypothesis}{Hypothesis}{Conditions}
\crefname{assumption}{Assumption}{Assumptions}
\crefname{lemma}{Lemma}{Lemmas}
\crefname{claim}{Claim}{Claims}

\Crefname{figure}{Figure}{Figures}

\crefformat{equation}{\textup{#2(#1)#3}}
\crefrangeformat{equation}{\textup{#3(#1)#4--#5(#2)#6}}
\crefmultiformat{equation}{\textup{#2(#1)#3}}{ and \textup{#2(#1)#3}}
{, \textup{#2(#1)#3}}{, and \textup{#2(#1)#3}}
\crefrangemultiformat{equation}{\textup{#3(#1)#4--#5(#2)#6}}%
{ and \textup{#3(#1)#4--#5(#2)#6}}{, \textup{#3(#1)#4--#5(#2)#6}}%
{, and \textup{#3(#1)#4--#5(#2)#6}}

\Crefformat{equation}{#2Equation~\textup{(#1)}#3}
\Crefrangeformat{equation}{Equations~\textup{#3(#1)#4--#5(#2)#6}}
\Crefmultiformat{equation}{Equations~\textup{#2(#1)#3}}{ and \textup{#2(#1)#3}}
{, \textup{#2(#1)#3}}{, and \textup{#2(#1)#3}}
\Crefrangemultiformat{equation}{Equations~\textup{#3(#1)#4--#5(#2)#6}}%
{ and \textup{#3(#1)#4--#5(#2)#6}}{, \textup{#3(#1)#4--#5(#2)#6}}%
{, and \textup{#3(#1)#4--#5(#2)#6}}

\crefdefaultlabelformat{#2\textup{#1}#3}

\newcommand{\grad}{\nabla}
\newcommand{\Cc}{C}
\newcommand{\lamstr}{\lambda^{\mspace{-2mu}*}}% lambda with corrected 'star'
\newcommand{\Psis}{\Psi_{\mspace{-2mu*}}}
\newcommand{\psis}{\psi^{}_*}

\newcommand{\process}[1]{{\{#1_t\}_{t\ge0}}}

\newcommand{\Ind}{\mathds{1}}           % indicator function
\newcommand{\transp}{^{\mathsf{T}}}     % transpose

\newcommand{\Exp}{{\mathbb{E}}}         % Expectation
\newcommand{\Prob}{{\mathbb{P}}}        % Probability
\newcommand{\RR}{{\mathbb R}}           % Real line
\newcommand{\Rd}{{{\mathbb R}^d}}
\newcommand{\NN}{{\mathbb N}}           % Natural numbers
\newcommand{\D}{\mathrm{d}}             % Differential
\newcommand{\E}{\mathrm{e}}             % Exponent
\newcommand{\df}{\coloneqq}             % define

\newcommand{\sB}{{\mathscr{B}}}         % Ball in \Rd
\newcommand{\Act}{\mathbb{U}}           % Action Set
\newcommand{\Uadm}{\mathfrak{U}}        % Admissible Controls
\newcommand{\Usm}{\mathfrak{U}_{\mathsf{sm}}}  %Stationary Markov controls
\newcommand{\Usms}{\mathfrak{U}^*_{\mathsf{sm}}}
\newcommand{\cPsv}{{\cP^{}_{\mspace{-3mu}*,v}}}
\newcommand{\cPov}{{\cP^{}_{\mspace{-3mu}\circ,v}}}

\newcommand{\Sob}{\mathscr{W}}          % Sobolev space
\newcommand{\Sobl}{\mathscr{W}_{\mathrm{loc}}}  %Sobolev space (local)

\newcommand{\Ag}{{\mathcal{A}}}
\newcommand{\fB}{\mathfrak{B}}          % Borel sets
\newcommand{\sF}{\mathfrak{F}}          % sigma field
\newcommand{\cG}{{\mathcal{G}}}         % linear operator
\newcommand{\cH}{{\mathcal{H}}}         % ergodic occupation measures
\newcommand{\Lp}{{L}}            % Lp
\newcommand{\Lg}{\mathcal{L}}           % Controlled extended generator
\newcommand{\cM}{\mathcal{M}}
\newcommand{\sM}{\mathscr{M}}           % Exponential Martingale
\newcommand{\cP}{{\mathcal{P}}}         % probability measures
\newcommand{\Lyap}{\mathscr{V}}         % Lyapunov function

\newcommand{\abs}[1]{\lvert#1\rvert}
\newcommand{\norm}[1]{\lVert#1\rVert}
\newcommand{\babs}[1]{\bigl\lvert#1\bigr\rvert}

\newcommand{\bnorm}[1]{\bigl\lVert#1\bigr\rVert}

\DeclareMathOperator*{\dist}{dist}
\DeclareMathOperator{\trace}{Tr}

\newcommand{\ttl}{\Large
A variational characterization of  the optimal\\[5pt]
exit rate for controlled diffusions}
\begin{document}
\title[On the optimal exit rate for controlled diffusions]{\ttl}

\author[Ari Arapostathis]{Ari Arapostathis$^\dag$}
\address{$^\dag$Department of Electrical and Computer Engineering,
The University of Texas at Austin, 2501 Speedway,
EER~7.824,
Austin, TX~78712}
\email{ari@utexas.edu} %\,\protect\orcidicon{0000-0003-2207-357X}}

\author[Vivek S. Borkar]{Vivek S. Borkar$^\ddag$}
\address{$^\ddag$Department of Electrical Engineering,
Indian Institute of Technology, Powai, Mumbai 400076, India}
\email{borkar@ee.iitb.ac.in} %\,\protect\orcidicon{0000-0003-0756-5402}}

%\date{\today}
%%%%%%%%%%%%%%%%%%%%%%%%%%%%%%%%%%%%%%%%%%%%%%%%%%%%%%%%%%%%%%%%%%%%%%%%%%%%%%%%
\begin{abstract}
The main result in this paper is a variational
formula for the exit rate from a bounded domain for a diffusion process
in terms of the stationary law of the diffusion constrained to remain in this
domain forever.
Related results on the geometric ergodicity of the controlled
$Q$-process are also presented.
\end{abstract}



\maketitle

%%%%%%%%%%%%%%%%%%%%%%%%%%%%%%%%%%%%%%%%%%%%%%%%%%%%%%%%%%%%%%%%%%%%%%%%%%%%%%%%
\section{Introduction}

A variational formulation for the principal eigenvalue of certain elliptic
operators associated with diffusion processes was given in the celebrated
article of Donsker and Varadhan \cite{DonVar-75}.
A discrete counterpart for Markov chains appears in \cite{DZ}. Subsequently, the
present authors and their collaborators extended it to controlled Markov chains
\cite{VenBor-17} and diffusions \cites{ABis-19,ABBK-19,ABK};
see also \cite{Survey} for an overview.
These articles also point out that these results are abstract counterparts of the
Collatz--Wielandt formula for the Perron--Frobenius eigenvalue of irreducible
non-negative matrices \cites{Collatz,Wielandt}.

The present article makes a connection between the study
of principal eigenvalues and a different strand of research.
This concerns quasi-stationarity which aims to study the `near-stationary' behavior
of a Markov process when it spends a long time in a sub-domain of its state space
before exiting from it \cites{Collet,Meleard-12}.
An important notion in this context is the law of the process conditioned on never
exiting a given domain.
This is the so called \textit{$Q$-process}, which in case of a diffusion, amounts
to adding a drift that explodes at the boundary of the domain so as to
confine the process to the domain in a precise manner.
The $Q$-process has been studied for various classes of Markov processes, a landmark
paper for the special case of diffusions being \cite{Pinsky-85}.
Some important recent contributions to the general case are
\cites{CV,Champagnat-17,Knobloch-10}.
As a part of our development, we re-prove some of their results for diffusions
using stochastic calculus based arguments.
Our main result is a new variational formula for the optimal eigenvalue of the
aforementioned controlled eigenvalue problem in terms of the associated $Q$-process.

The principal eigenvalue is also the exit rate from the domain, that is, the asymptotic
exponential decay rate of the tail of the distribution of the first exit time $\uptau$
from the domain, given by $-\lim_{t\uparrow\infty}\frac{1}{t}\log P(\uptau > t)$.
Exit time statistics and related themes have been of independent interest, particularly
in the small noise limit, and have been extensively studied, e.g., in
\cites{Befekadu-15,Fleming-77,FlemJam-92,Fleming-85,Fleming-81,Gong-88}.
Our focus, however, is different from that of these works.

The article is organized as follows. The next subsection introduces the key notation.
\Cref{S2} introduces the controlled eigenvalue problem associated with a
diffusion in a bounded domain.
\Cref{S3} introduces the $Q$-process whose law agrees
with the law of the original diffusion
conditioned on never exiting the prescribed domain.
It also develops some key properties of this process such as ergodicity and the
associated invariant probability measure.
\Cref{S4} states and proves our main  result, a new representation theorem for the
optimal eigenvalue in terms of the $Q$-process.

%%%%%%%%%%%%%%%%%%%%%%%%%%%%%%%%%%%%%%%%%%%%%%%%%%%%%%%%%%%%%%%%%%%%%%%%%%%%%%%%
\subsection{Notation}\label{Snot}
We denote by $\uptau(A)$ the \emph{first exit time} of the process
$\process{X}$ from the set $A\subset\RR^{d}$, defined by
\begin{equation}\label{E-uptau}
\uptau(A) \,\df\, \inf\,\{t>0\,\colon X_{t}\not\in A\}\,.
\end{equation}

The complement, closure, and boundary
of a set $A\subset\Rd$ are denoted
by $A^{c}$, $\Bar{A}$ and $\partial A$, respectively, and  $\Ind_A$
denotes its indicator function.
Given $a,b\in\RR$, the minimum (maximum) is denoted by $a\wedge b$
($a\vee b$), respectively.
The inner product of two vectors $x$ and $y$ in $\Rd$ is denoted
as $\langle x,y\rangle$, $\abs{\,\cdot\,}$ denotes
the Euclidean norm, $x\transp$ stands for
the transpose of $x$, and
$\trace S$ denotes the trace of a square matrix $S$.

The term \emph{domain} in $\RR^{d}$
refers to a nonempty, connected open subset of the Euclidean space $\RR^{d}$.
For a domain $D\subset\RR^{d}$,
the space $\Cc^{k}(D)$ ($\Cc^{k}_{b}(D)$), $k\ge 0$,
refers to the class of all real-valued functions on $D$ whose partial
derivatives up to order $k$ exist and are continuous (and bounded),
$\Cc_{\mathrm{c}}^k(D)$ denotes its subset
 consisting of functions that have compact support,
and $\Cc_0^k(D)$ the closure of $\Cc_{\mathrm{c}}^k(D)$.
Also $C^k_+(D)$ denotes the subspace of $C^k(D)$ consisting of those
functions that are positive on $D$.
The space $\Lp^{p}(D)$, $p\in[1,\infty)$, stands for the Banach space
of (equivalence classes of) measurable functions $f$ satisfying
$\int_{D} \abs{f(x)}^{p}\,\D{x}<\infty$, and $\Lp^{\infty}(D)$ is the
Banach space of functions that are essentially bounded in $D$.
The standard Sobolev space of functions on $D$ whose generalized
derivatives up to order $k$ are in $\Lp^{p}(D)$, equipped with its natural
norm, is denoted by $\Sob^{k,p}(D)$, $k\ge0$, $p\ge1$.
In general, if $\mathcal{X}$ is a space of real-valued functions on $D$,
the space
$\mathcal{X}_{\mathrm{loc}}$ consists of all functions $f$ such that
$f\varphi\in\mathcal{X}$ for every $\varphi\in\Cc_{\mathrm{c}}(\mathcal{X})$.
This defines $\Sobl^{k,p}(D)$.

We say that a continuous function $f\colon D\to\RR$ is \emph{inf-compact}
if the sublevel set $\{x\in D\colon f(x)\le C\}$ is compact (or empty)
for any $C\in\RR$.

%%%%%%%%%%%%%%%%%%%%%%%%%%%%%%%%%%%%%%%%%%%%%%%%%%%%%%%%%%%%%%%%%%%%%%%%%%%%%%%%
\section{The controlled exit rate problem}\label{S2}

We begin by recalling the controlled eigenvalue problem from \cite{BB}.
We consider a controlled diffusion given by the $d$-dimensional It\^o stochastic
differential equation (SDE)
\begin{equation}\label{sde}
X_t \,=\, X_0 + \int_0^t m(X_s, U_s)\,\D{s} + \int_0^t\upsigma(X_s)\,\D{W}_s\,,
\quad t \ge 0\,,
\end{equation}
where:
\begin{enumerate}
\item $m \colon\Rd\times\Act \mapsto \Rd$ for a prescribed compact metric `control'
space $\Act$, is continuous and locally Lipschitz in its first argument uniformly
with respect to the second, and satisfies, for some constant $C$,
\begin{equation*}
\langle m(x,u),x\rangle \,\le\, C \abs{x}^2\quad\forall\,x,u\in\Rd\times\Act\,;
\end{equation*}

\item $\upsigma \colon \Rd \mapsto \Rd$ is Lipschitz and satisfies:
\begin{equation*}
\abs{\upsigma\transp(x)y}^2 \,\ge\,
c_0\,\abs{y}^2 \quad \forall\, x, y \in \Rd\,,
\end{equation*}
for some $c_0 > 0\,$;

\item $X_0$ is prescribed in law with bounded moments;

\item $\process{W}$ is a standard Brownian motion in $\Rd$ independent of $X_0\,$;

\item $\process{U}$ is a $\Act$-valued process with measurable paths, satisfying
the `non-anticipativity condition':
for all $t > s \ge 0$, $W_t - W_s$ is independent of $\sF_s$, which
is defined as the right-continuous completion of
$\sigma(X_r, U_r, r \le s)$.
We call such $\process{U}$ \emph{admissible}, and let $\Uadm$ denote
the class of these controls.
\end{enumerate}

We shall use the relaxed  control formulation,
that is, $\Act = \cP(\Act_0)$ where $\Act_0$ is compact metric and $\cP(\Act)$
is the compact Polish space of probability measures on $\Act$
 with the Prokhorov topology, and furthermore, $m$ is of the form
$$m(x, u) \,=\, \int_{\Act_0} m_0(x, y)u(\D y)$$
for some $m_0\colon \Rd\times \Act_0 \mapsto \Rd$ which is continuous and Lipschitz
in its first argument uniformly with respect to the second.
We also define the special control class of \emph{stationary controls} wherein
$U_t = v(X_t)$ for some measurable $v \colon \Rd \mapsto\Act$, identified with
the map $v$ by standard abuse of terminology and termed
\emph{stationary control policy}.
We let $\Usm$ denote the class of these policies.
Under any $v\in\Usm$, the SDE in \cref{sde} has a unique strong solution,
and this is a strong Markov process \cite{Gyongy-96}.
We let $\Prob^v_x$ and $\Exp^v_x$ denote the probability measure
and the expectation operator, respectively, on the canonical space of the process
$\process{X}$ controlled by $v\in\Usm$ and with initial condition $X_0=x$.
We extend this definition to the class of admissible controls, and
use $\Prob^U_x$ and $\Exp^U_x$ for $U\in\Uadm$.
For $v\in\Usm$, we also use the simplified notation
\begin{equation*}
m_v(x)\,\df\, m\bigl(x,v(x)\bigr)\,.
\end{equation*}

We introduce the following notation for the controlled extended generator
of $\process{X}$:
\begin{equation*}
\Lg f(x,u) \,\df\,  \frac{1}{2}\trace\left(a(x)\nabla^2f(x)\right)
+ \bigl\langle m(x,u),\nabla f(x)\bigr\rangle
\quad\text{for\ } f \in C^2(\Rd)\,,
\end{equation*}
with
\begin{equation*}
a(x) \,\df\, \upsigma(x)\upsigma\transp(x)\,,\quad x\in\Rd\,.
\end{equation*}
Under a stationary policy $v$ as above,
we denote $\Lg f\bigl(x,v(x)\bigr)$ as $\Lg_vf(x)$.

Let $D$ be a bounded domain with $\Cc^{2,1}$ boundary,
which is kept fixed throughout the paper,
and $\uptau\equiv\uptau(D)$, the first exit time from $D$
(see \cref{E-uptau}).
In this work, we consider the problem of minimizing
the \emph{rate of exit} from $D$ over all admissible controls, that is
\begin{equation}\label{E-beta*}
\beta^*(x) \,\df\, \inf_{U\in\Uadm}\,\biggl(-\limsup_{T\nearrow\infty}\,
\frac{1}{T}\,\log \Prob_x^U(\tau > T)\biggr)\,.
\end{equation}
It turns out that $\beta^*$ is independent of $x$.
Note that
\begin{equation*}
\Prob_x^U(\tau > T) \,=\,\Prob_x^U\bigl(X_t \in D\,,\; t \in [0,T]\bigr)\,.
\end{equation*}

Consider the process under a given control $v\in\Usm$.
Then it is well known \cite{BB} that the rate of exit
\begin{equation}\label{E-betav}
\beta_v \,\df\, -\limsup_{T\nearrow\infty}\,
\frac{1}{T}\,\log \Prob_x^v(\tau > T)
\end{equation}
is equal to the principal eigenvalue of the operator $\Lg_v$ on $D$, which is
defined as
\begin{equation}\label{E-lambdav}
\lambda_v \,\df\, \sup\, \bigl\{\lambda\in\RR\,\colon
\exists\, \phi\in\Sobl^{2,d}(D),\, \phi>0\text{\ in\ }D,\, \Lg_v\phi+\lambda\phi\le0
\text{\ in\ }D\bigr\}\,.
\end{equation}
It is also well known \cite{BNV} that $\lambda_v$ is a simple eigenvalue, and its
eigenvector, or eigenfunction, is the unique solution $\Psi_v$
in $\Sobl^{2,p}(D)\cap\Cc(\overline{D})$, for any $p\ge d$, to the Dirichlet problem
\begin{equation}\label{E-Dirv}
\begin{aligned}
\Lg_v\Psi_v(x) + \lambda_v\Psi_v(x) &\,=\, 0\text{\ \ a.e.\ }x\text{\ in\ } D\,,\\
\Psi_v &\,>\,0\text{\ \ in\ }D\,,\\
 \Psi_v &\,=\,0\text{\ \ on\ }\partial D\,.
\end{aligned}
\end{equation}
Uniqueness of the eigenfunction is, of course, up to scalar multiplication with
a positive constant, and keeping that in mind, we use
the term unique without any reference to a specific normalization.
For such Dirichlet eigenvalue problems, we refer to $(\lambda_v,\Psi_v)$ as
an \emph{eigenpair}.

Going back to the optimal rate $\beta^*$ in \cref{E-beta*}, we define the
semilinear operator $\cG^*$ by
\begin{equation}\label{E-cGmax}
\cG^* f(x) \,\df\, \sup_{u\in\Act}\, \Lg f(x,u)\,,
\end{equation}
and denote its principal eigenvalue in $D$ as $\lamstr$, which is
defined exactly as in \cref{E-lambdav} by replacing $\Lg_v$ with $\cG^*$.
Then, as shown in \cite{BB}, $\beta^*=\lamstr = \inf_{v\in\Usm}\lambda_v$,
and again, $\lamstr$ is a simple eigenvalue, and its eigenfunction
$\Psi^*\in\Cc^{2}(\overline{D})$
is the unique solution to the Dirichlet problem
\begin{equation}\label{E-Dir*}
\begin{aligned}
\max_{u\in\Act}\,
\Lg\Psi^*(x,u) + \lamstr\Psi^*(x) &\,=\, 0\text{\ \ a.e.\ }x\text{\ in\ } D\,,\\
\Psi^* \,>\,0\text{\ \ in\ }D\,,&\qquad
\Psi^* \,=\,0\text{\ \ on\ }\partial D\,.
\end{aligned}
\end{equation}

We note here that the Lemma and Theorem in \cite{BB} actually state that
$\lamstr = \inf_{v\in\Usm} \lambda_v$, rather than $\lamstr=\beta^*$
which was claimed above.
However, a close inspection shows that using the same argument in the first
part of their proof,
we obtain the upper bound $\lamstr\le\inf_{U\in\Uadm}\beta_U$, with
$\beta_U$ defined as in \cref{E-betav}, by replacing $v$ with $U\in\Uadm$.
So indeed, $\lamstr=\beta^*$ by the results in \cite{BB}.
Let $\Usms$ denote the set of measurable selectors from the minimizer
in \cref{E-Dir*}.
This set is non-empty by \cite{Benes-70}.
Then, as shown in the Theorem in \cite{BB}, $v\in\Usms$ if and only if
$\beta_v=\beta^*$.

%%%%%%%%%%%%%%%%%%%%%%%%%%%%%%%%%%%%%%%%%%%%%%%%%%%%%%%%%%%%%%%%%%%%%%%%%%%%%%%%
\section{The \texorpdfstring{$Q$}{}-process}\label{S3}

Let $v$ be a generic element of $\Usm$.
Define $\psi_v\df\log\Psi_v$, with $\Psi_v$ the eigenfunction in \cref{E-Dirv}.
The $Q$-process $\process{\widetilde{X}}$ associated with the process
$\process{X}$ in \cref{sde}
on the domain $D$
is given by the It\^o stochastic differential equation
\begin{equation}\label{E-Q}
\D\widetilde{X}_t \,=\, \widetilde{m}_v(\widetilde{X}_t)\,\D{t}
 + \upsigma(\widetilde{X}_t)\,\D\widetilde{W}_t\,,
\end{equation}
with
\begin{equation*}
\widetilde{m}_v(x) \,\df\, m_v(x) + a(x)\nabla\psi_v(x)\,.
\end{equation*}
Note that its extended generator $\widetilde\Lg_v$ satisfies
\begin{equation}\label{E-TLgv}
\widetilde\Lg_v f \,\df\, \frac{1}{2} \trace\bigl(a\nabla^2 f\bigr)
+ \langle\widetilde{m}_v,\nabla f\rangle
\,=\, \Lg_v f + \langle a\grad\psi_v,\grad f\rangle\,.
\end{equation}
By \cref{E-Dirv,E-Dir*}, we also have the identities
\begin{equation}\label{E-eigenA}
\Lg_v\psi_v(x)
+\frac{1}{2}\babs{\upsigma\transp(x)\nabla\psi_v(x)}^2\,=\,
\widetilde\Lg_v\psi_v(x)
-\frac{1}{2}\babs{\upsigma\transp(x)\nabla\psi_v(x)}^2 \,=\, -\lambda_v\,,\quad x\in D\,,
\end{equation}
and
\begin{equation}\label{E-eigenB}
\max_{u\in\Act}\, \left(\Lg\psi^*(x,u)
+\frac{1}{2}\babs{\upsigma\transp(x)\nabla\psi^*(x)}^2\right)
\,=\, -\lamstr\,,\quad x\in D\,,
\end{equation}
with $\psi^*\df\log\Psi^*$.
We often use both operators $\Lg_v$ and $\widetilde\Lg_v$.

%%%%%%%%%%%%%%%%%%%%%%%%%%%%%%%%%%%%%%%%%%%%%%%%%%%%%%%%%%%%%%%%%%%%%%%%%%%%%%%%
\begin{remark}\label{R3.1}
Since we often refer to the results in
\cites{ABis-18,ABS-19,ABis-19,ABBK-19,ABis-19b},
we want to caution the reader that in all these papers, the principal eigenvalues
are defined with the opposite sign.
In particular, with respect to its dependence on the domain $D$,
$\lambda_v$ in \cref{E-lambdav} satisfies $\lambda_v(D)>\lambda_v(D')$
whenever $\overline{D}\subset D'$.
\end{remark}

\begin{remark}\label{R3.2}
When applying It\^o's formula in equations such as \cref{E-eigenA},
in order to localize the martingales, we use the stopping time
$\uptau_\epsilon$ which is defined as the first exit time from the set
\begin{equation*}
D_\epsilon\,\df\, \bigl\{x\in D\,\colon \dist(x,\partial D)>\epsilon\bigr\}\,,
\end{equation*}
where `$\dist$' denotes the Euclidean distance.
\end{remark}

%%%%%%%%%%%%%%%%%%%%%%%%%%%%%%%%%%%%%%%%%%%%%%%%%%%%%%%%%%%%%%%%%%%%%%%%%%%%%%%%
\subsection{Ergodicity of the \texorpdfstring{$Q$}{}-process}

It is well known that the $Q$-process is confined to $D$ and is geometrically
ergodic \cites{Pinsky-85,CV}.  This is in fact true for a more general class
of models.  In the theorem which follows we give an independent proof of
this fact for diffusions using PDE theory and stochastic calculus.
In the process, we present some important techniques concerning
 the eigenvalue problem. An important inequality which we use,
 and which can be easily verified, is the following:
\begin{equation}\label{E-nice}
\text{if\ \ } \Lg_v\Phi + F \Phi\,\le\,0\text{\ \ in\ }D\,,\quad\text{then\ \ }
\widetilde\Lg_v\Bigl(\frac{\Phi}{\Psi}\Bigr) \,\le\, (\lambda_v-F) \frac{\Phi}{\Psi_v}
\text{\ \ in\ }D\,.
\end{equation}
We also use the notation $A\Subset D$ to indicate that $\Bar{A}\subset D$.

%%%%%%%%%%%%%%%%%%%%%%%%%%%%%%%%%%%%%%%%%%%%%%%%%%%%%%%%%%%%%%%%%%%%%%%%%%%%%%%%
\begin{theorem}\label{T3.1}
For any $v\in\Usm$, the $Q$-process in \cref{E-Q}
is confined to $D$ and is geometrically ergodic.
In particular, there exists $\Lyap_v\in\Sobl^{2,p}(D)$, $p>d$,
a compact set $K\subset D$, and positive constants $C_v$ and $\rho_v$, such that
\begin{equation}\label{ET3.1A}
\widetilde\Lg_v \Lyap_v(x) \,\le\, C_v\Ind_{K}(x) -  \rho_v\Lyap_v(x)
\qquad\forall\,x\in D\,,
\end{equation}
and $\Lyap_v\Psi_v$ is bounded from below away from $0$ in $D$.
As a result, $\E^{-\psi_v}$ is integrable under the invariant
probability measure $\Tilde\mu_v$ of the $Q$-process.
Moreover,
\begin{equation}\label{ET3.1B}
\lambda_v \,=\, \frac{1}{2} \int_{D} \babs{\upsigma\transp(x)\nabla\psi_v(x)}^2\,
\Tilde\mu_v(\D{x})\,.
\end{equation}
\end{theorem}

\begin{proof}
In the interest of economy
of notation, we drop the explicit dependence on $v$ from
$\lambda_v$, $\Lg_v$, and $\Psi_v$.
Let $\sB\Subset D$ be a nonempty open set.
Then  the principal eigenvalue   of the
operator $\Lg-\Ind_\sB$ on $D$,
denoted as $\lambda_D(\Lg-\Ind_\sB)$,
satisfies $\lambda_D(\Lg-\Ind_\sB)>\lambda$ \cite[Lemma~2.1\,(b)]{ABS-19}.
Hence, by the monotonicity and continuity of the principal eigenvalue
with respect to the domain \cite[Lemma~2.2\,(a)]{ABis-18}
(see also the more general result in \cite[Corollary~2.1]{ABis-19b}),
there exists a bounded $\Cc^{2,1}$ domain $D'\Supset D$, such that
\begin{equation*}
\lambda_D(\Lg-\Ind_\sB)\,>\,\lambda_{D'}(\Lg-\Ind_\sB)\,>\,\lambda\,.
\end{equation*}
Note that \cref{R3.1} applies to these assertions.
Let $\Phi$ denote the principal eigenfunction of $\Lg-\Ind_\sB$ on $D'$.
Then we have
\begin{equation*}
\Lg \Phi +\bigl(\lambda_{D'}(\Lg-\Ind_\sB)-\Ind_\sB\bigr)\Phi\,=\,0
\quad\text{in\ }D'\,.
\end{equation*}
Therefore, by \cref{E-nice}, we obtain
\begin{equation}\label{PT3.1B}
\widetilde\Lg\Bigl(\frac{\Phi}{\Psi}\Bigr) \,\le\,
\Bigl(\Ind_\sB+\lambda-\lambda_{D'}(\Lg-\Ind_\sB)\Bigr)\,\frac{\Phi}{\Psi}\,.
\end{equation}
Since $\lambda_D(\Lg-\Ind_\sB)>\lambda$, and
$\inf_{D}\,\Phi>0$ by the continuity and positivity of
$\Phi$ on $D'$, the Foster--Lyapunov inequality
in \cref{ET3.1A} is implied by \cref{PT3.1B}
with $\Lyap_v \df \frac{\Phi}{\Psi}$.
It is clear that $\Lyap_v(x)\to\infty$
as $x\to\partial D$.
Geometric ergodicity of the $Q$-process then follows by
\cite[Theorem~5.2]{Down-Meyn-Tweedie-1995}.

It is fairly standard to show that \cref{ET3.1A} implies that
the $Q$-process is confined to $D$.
Indeed, using the It\^o formula,
and with $\uptau_\epsilon$ and $D_\epsilon$ as in \cref{R3.2}, we obtain
\begin{equation*}
\widetilde\Exp_x^v\bigl[\Lyap_v(\widetilde{X}_{t\wedge\uptau_\epsilon})\bigr]
\,\le\, \Lyap_v(x) + C_v t\,,
\end{equation*}
which implies by the Markov inequality that
\begin{equation*}
\widetilde\Prob_x(\uptau_\epsilon<t)\,\le\,
\Bigl(\inf_{\partial D_\epsilon}\,\Lyap_v\Bigl)^{-1}\bigl(\Lyap_v(x)+C_v t\bigr)\,.
\end{equation*}
Thus, since $\Lyap_v$ is inf-compact, we have
$\widetilde\Prob_x(\uptau_\epsilon<t)\to0$
as $\epsilon\searrow0$ for all $x\in D$ and $t>0$.

Since $\Lyap_v\ge c\, \E^{-\psi}$ for some positive constant $c$,
and $\Lyap_v$ is inf-compact, it follows that $\frac{\abs{\psi(x)}}{\Lyap_v(x)}$ tends
to $0$ as $\dist(x,\partial D)\to0$.
By \cite[Lemma~3.7.2]{ABG}, then we have
\begin{equation}\label{PT3.1C}
\widetilde\Exp_x\bigl[\psi(X_{t\wedge\uptau_\epsilon})\bigr]
\,\xrightarrow[\epsilon\searrow0]{}\,
\widetilde\Exp_x\bigl[\psi(X_{t})\bigr]\,,\quad\text{and\ \ }
\lim_{t\to\infty}\,\frac{1}{t}\,\widetilde\Exp_x\bigl[\abs{\psi(X_{t})}\bigr]\,=\,0
\end{equation}
for all $x\in D$.
Thus applying It\^o's formula to $\psi$,  using \cref{E-eigenA},
\cref{R3.2} and \cref{PT3.1C}, we obtain \cref{ET3.1B}.
\end{proof}

The method of proof of \cref{T3.1} leads to a another interesting result.
Let $\cG_*$ denote the minimal operator (compare with \cref{E-cGmax})
\begin{equation}\label{E-cGmin}
\cG_* f(x) \,\df\, \inf_{u\in\Act}\, \Lg f(x,u)\,,
\end{equation}
and denote its principal eigenpair in $D$ as $(\lambda_*,\Psis)$.
It is clear that $\lambda_*\ge\lamstr$, and that for a generic drift $m$
we should have strict inequality.
Let $\psis=\log\Psis$, which satisfies
\begin{equation*}
\max_{u\in\Act}\, \Lg\psis(x,u)
+\frac{1}{2}\babs{\upsigma\transp(x)\nabla\psis(x)}^2 \,=\, -\lamstr\,,\quad x\in D\,.
\end{equation*}
Consider the controlled diffusion
\begin{equation}\label{E-Qmin}
\D{Y}_t \,=\, \bigl(m(Y_t,U_t) + a(Y_t)\grad\psis(Y_t)\bigr)\,\D{t}
 + \upsigma(Y_t)\,\D{W}_t\,.
\end{equation}
Denote its extended controlled generator as $\Lg^{\psis}$, that is,
\begin{equation*}
\Lg^{\psis} f(x,u) \,\df\,  \frac{1}{2}\trace\left(a(x)\nabla^2f(x)\right)
+ \bigl\langle m(x,u)+a(x)\grad\psis(x),\nabla f(x)\bigr\rangle
\quad\text{for\ } f \in C^2(D)\,.
\end{equation*}
As done earlier, under $v\in\Usm$, we denote
 $\Lg^{\psis} f\bigl(x,v(x)\bigr)$ as $\Lg^{\psis}_vf(x)$.

 We have the following result, which in a way extends \cref{T3.1}.

%%%%%%%%%%%%%%%%%%%%%%%%%%%%%%%%%%%%%%%%%%%%%%%%%%%%%%%%%%%%%%%%%%%%%%%%%%%%%%%%
\begin{theorem}%\label{T3.2}
The process $\process{Y}$ in \cref{E-Qmin} is confined to $D$,
and is geometrically ergodic
uniformly over $v\in\Usm$.
In particular, there exists $\widehat\Lyap\in\Cc^2(D)$,
a compact set $K\subset D$, and positive constants $C$ and $\rho$, such that
\begin{equation}\label{ET3.2A}
\Lg^{\psis} \widehat\Lyap(x,u) \,\le\, C\Ind_{K}(x) -  \rho\widehat\Lyap(x)
\qquad\forall\,(x,u)\in D\times\Act\,,
\end{equation}
and $\widehat\Lyap\,\Psis$ is bounded from below away from $0$ in $D$.
As a result, $\E^{-\psis}$ is integrable under the invariant
probability measure $\Breve\mu_v$ of $\process{Y}$ for any $v\in\Usm$.
Moreover,
\begin{equation}\label{ET3.2B}
\lambda_* \,\ge\, \frac{1}{2} \int_{D} \babs{\upsigma\transp(x)\nabla\psis(x)}^2\,
\Breve\mu_v(\D{x})\qquad\forall\,v\in\Usm\,.
\end{equation}
\end{theorem}

\begin{proof}
Let $D_\epsilon$ be as in \cref{R3.2}.
It is clear that $\lambda_D(\cG^* - 2\lambda_* \Ind_{D_\epsilon})$
is nondecreasing in $\epsilon$, and
converges to $\lamstr+2\lambda_*$ as $\epsilon\searrow0$.
Therefore, we can select some $\epsilon>0$ and a smooth
function $h\colon D\to[0,1]$ which is compactly supported in $D$ and
equals $1$ on $D_\epsilon$, such that
$\lambda_D(\cG^* - 2\lambda_*h)> \lamstr+\lambda_*$.
Thus, by the monotonicity of the principal eigenvalue
with respect to the domain, we can select a bounded $\Cc^{2,1}$ domain $D'\Supset D$
as in the proof of \cref{T3.1},
such that $\lambda_{D'}(\cG^* - 2\lambda_*h)> \lamstr+\lambda_*$.
Let $\Phi\in\Cc^2(D')$ denote the associated principal eigenfunction in $D'$.
Thus we have
\begin{equation*}
\cG^*\Phi +(\lamstr+\lambda_*-2\lambda_*h)\Phi\,\le\,0\quad\text{in\ }D'\,,
\end{equation*}
which of course implies that
\begin{equation}\label{PT3.2A}
\Lg\Phi(x,u) + \bigl(\lamstr+\lambda_*-2\lambda_*h(x)\bigr)\Phi(x)\,\le\,0
\qquad\forall\,(x,u)\in D\times\Act\,.
\end{equation}
On the other hand, using the eigenfunction $\Psis$ of
the operator $\cG_*$ in \cref{E-cGmin}, we obtain
\begin{equation}\label{PT3.2B}
\Lg\Psis(x,u) + \lambda_*\Psis(x)\,\ge\,0
\qquad\forall\,(x,u)\in D\times\Act\,.
\end{equation}
Combining \cref{PT3.2A,PT3.2B}, and using \cref{E-nice}, we get
\begin{equation*}
\Lg^{\psis}\Bigl(\frac{\Phi}{\Psis}\Bigr)(x,u) \,\le\,
\bigl(2\lambda_*h(x)-\lamstr)\,\frac{\Phi}{\Psis}(x)\,.
\end{equation*}
This establishes \cref{ET3.2A}, from which the integrability of $\E^{-\psis}$
and the property that the process is confined to $D$ follow as in the proof of the preceding theorem.
Also, \cref{ET3.2A} implies uniform geometric ergodicity.
This follows exactly as in the proof of \cite[Theorem~2.1\,(b)]{AHP18}
(see also \cite[Lemma~2.5.5]{ABG}).

In analogy to \cref{E-eigenA} we have
\begin{equation*}
\Lg^{\psis}_v\psis(x) -\frac{1}{2}\babs{\upsigma\transp(x)\nabla\psis(x)}^2
\,\ge\, -\lambda_v\,,\quad x\in D\,,
\end{equation*}
from which \cref{ET3.2B} follows.
This completes the proof.
\end{proof}

%%%%%%%%%%%%%%%%%%%%%%%%%%%%%%%%%%%%%%%%%%%%%%%%%%%%%%%%%%%%%%%%%%%%%%%%%%%%%%%%
\subsection{The quasi-stationary distribution}

An application of the It\^o formula to the first equation in \cref{E-Dirv}
shows that
\begin{equation*}
\E^{-\lambda_v t}\E^{\psi_v(x)} \,=\, \Exp^v_x\Bigl[ \E^{\psi_v(X_t)}\,
\Ind_{\{t<\uptau\}}\Bigr]\,,\quad x\in D\,.
\end{equation*}
Let
\begin{equation*}
\sM_t\,\df\,\exp\biggl(\int_0^t \bigl\langle
\upsigma\transp(X_s)\grad\psi_v(X_s),\D W_s\bigr\rangle\,\D{s}
- \frac{1}{2}
\int_0^t \babs{\upsigma\transp(X_s)\nabla\psi_v(X_s)}^2\,\D{s}\biggr)\,.
\end{equation*}
Let $\widetilde\Prob^v_x$ and $\widetilde\Exp^v_x$ denote the probability measure
and the expectation operator, respectively, on the canonical space of the process
$\process{\widetilde{X}}$ in \cref{E-Q} controlled by $v\in\Usm$,
 and with initial condition $X_0=x$.
Then,
 for any bounded function $g$ which is compactly supported in $D$,
and $\uptau_\epsilon$ as in \cref{R3.2}, we get
\begin{equation*}%\label{E-GirC}
\begin{aligned}
\Exp^v_x\Bigl[g(X_t)\,\Ind_{\{t < \uptau_\epsilon\}}\Bigr]
&\,=\,  \Exp^v_x \Bigl[ g(X_{t\wedge\uptau_\epsilon})\,
\Ind_{\{t < \uptau_\epsilon\}}\Bigr]\\[3pt]
&\,=\, \Exp^v_x\Bigl[\E^{-\lambda_v (t\wedge\uptau_\epsilon)}g(X_{t\wedge\uptau_\epsilon})
\exp\bigl(-\psi_v(X_{t\wedge\uptau_\epsilon}) + \psi_v(x)\Bigr)
\sM_{t\wedge\uptau_\epsilon}\,\Ind_{\{t < \uptau_\epsilon\}}\Bigr]\\[3pt]
&\,=\, \widetilde\Exp^v_x\Bigl[\E^{-\lambda_v (t\wedge\uptau_\epsilon)}
\,g(\widetilde{X}_{t\wedge\uptau_\epsilon}) \,
\exp\bigl(-\psi_v(\widetilde{X}_{t\wedge\uptau_\epsilon}) + \psi_v(x)\bigr)
\,\Ind_{\{t < \uptau_\epsilon\}}\Bigr]\\[3pt]
&\,=\, \E^{-\lambda_v t} \E^{\psi_v(x)}\,\widetilde\Exp^v_x \Bigl[g(\widetilde{X}_t) \,
\exp\bigl(-\psi_v(\widetilde{X}_t)\bigr)\,\Ind_{\{t < \uptau_\epsilon\}}\Bigr]\,,
\end{aligned}
\end{equation*}
where in the third equality we use Girsanov's theorem.
Letting $\epsilon\searrow0$, we obtain
\begin{equation}\label{E-GirC}
\begin{aligned}
\Exp^v_x\Bigl[g(X_t)\,\Ind_{\{t < \uptau\}}\Bigr]
&\,=\, \E^{-\lambda_v t} \E^{\psi_v(x)}\,\widetilde\Exp^v_x \Bigl[g(\widetilde{X}_t) \,
\exp\bigl(-\psi_v(\widetilde{X}_t)\bigr)\,\Ind_{\{t < \uptau\}}\Bigr]\\[3pt]
&\,=\, \E^{-\lambda_v t} \E^{\psi_v(x)}\,\widetilde\Exp^v_x \Bigl[g(\widetilde{X}_t) \,
\exp\bigl(-\psi_v(\widetilde{X}_t)\bigr)\Bigr]\,,
\end{aligned}
\end{equation}
where we drop
the term $\Ind_{\{t < \uptau\}}$ in the second equality since the process
$\process{\widetilde{X}}$ is confined to $D$.
Then by monotone convergence we can extend \cref{E-GirC} to all
nonnegative bounded functions $g$.
Since $\Psi_v$ vanishes on $\partial{D}$, we also obtain from \cref{E-GirC}
using monotone convergence that
\begin{equation*}
\Exp^v_x\Bigl[h(X_t)\exp\bigl(\psi_v(X_t)\bigr)\,\Ind_{\{t < \uptau\}}\Bigr]
\,=\,
\E^{-\lambda_v t} \E^{\psi_v(x)}\,\widetilde\Exp^v_x
\Bigl[h(\widetilde{X}_t)\Bigr]
\end{equation*}
for all bounded functions $h$.

By \cref{E-GirC}, we obtain
\begin{equation}\label{E-GirE}
\Exp^v_x\bigl[g(X_t)\mid  t < \uptau\bigr]
\,=\, \frac{\Exp^v_x\Bigl[g(X_t)\,\Ind_{\{t < \uptau\}}\Bigr]}
{\Exp^v_x\Bigl[\Ind_{\{t < \uptau\}}\Bigr]}
\,=\, \frac{\widetilde\Exp^v_x \Bigl[g(\widetilde{X}_t) \,
\exp\bigl(-\psi_v(\widetilde{X}_t)\bigr)\Bigr]}
{\widetilde\Exp^v_x \Bigl[
\exp\bigl(-\psi_v(\widetilde{X}_t)\bigr)\Bigr]}\,.
\end{equation}
Since the $Q$-process is
ergodic, and $\E^{-\psi_v}$ is integrable under its invariant distribution
$\Tilde\mu_v$ by \cref{T3.1}, then
taking limits in \cref{E-GirE} and using \cite[Theorem~4.12]{Ichihara-12},
we obtain
\begin{equation}\label{E-GirF}
\lim_{t\to\infty}\,\Exp^v_x \bigl[ g(X_t)\bigm| t < \uptau\bigr]
\,=\, \frac{\int_{D} g(y) \E^{-\psi_v(y)}\,\Tilde\mu_v(\D{y})}
{\int_{D} \E^{-\psi_v(y)}\,\Tilde\mu_v(\D{y})}\,.
\end{equation}

We have the following theorem.
%%%%%%%%%%%%%%%%%%%%%%%%%%%%%%%%%%%%%%%%%%%%%%%%%%%%%%%%%%%%%%%%%%%%%%%%%%%%%%%%
\begin{theorem}\label{T3.3}
It holds that
\begin{equation}\label{ET3.3A}
\alpha_v(A)\,\df\, \lim_{t\to\infty}\, \Prob^v_x(X_t\in A\,|\,t<\uptau)
\,=\, \frac{\int_{D} \Ind_A(y)\, \E^{-\psi_v(y)}\,\Tilde\mu_v(\D{y})}
{\int_{D} \E^{-\psi_v(y)}\,\Tilde\mu_v(\D{y})}\,,
\quad A\in\fB(D)\,,
\end{equation}
where $\fB(D)$ denotes the class of Borel sets in $D$.
In addition,
\begin{equation}\label{ET3.3B}
\Tilde\mu_v(A) \,=\,\frac{\int_{D} \Ind_A(y)\, \E^{\psi_v(y)}\,\alpha_v(\D{y})}
{\int_{D} \E^{\psi_v(y)}\,\alpha_v(\D{y})}\,,
\end{equation}
and
\begin{equation}\label{ET3.3C}
\lim_{t\to\infty}\, \E^{\lambda_v t}\,\Prob^v_x(t<\uptau)
\,=\, \E^{\psi_v(x)}\int_{D} \E^{-\psi_v(y)}\,\Tilde\mu_v(\D{y})\,.
\end{equation}
\end{theorem}

\begin{proof}
\Cref{ET3.3A} is a direct consequence of \cref{E-GirF}
with $g(\cdot) \equiv 1$, and
\cref{ET3.3B} follows by inverting \cref{ET3.3A}.
Lastly, \cref{ET3.3C} follows by taking limits in \cref{E-GirC}
with $g(\cdot) \equiv 1$.
\end{proof}

The probability measure $\alpha_v$ in \cref{T3.3} is known in the
literature as a \emph{quasi-stationary distribution} \cites{CV,Champagnat-17,Meleard-12}.
\Cref{ET3.3B} should be compared with \cite[Theorem~1.3]{CV},
and \cref{ET3.3C} with \cite[Proposition~1.2]{CV}.

The rate of convergence in \cref{ET3.3A} has been studied extensively
in the literature (see the results and discussion in \cite{CV}).
It follows from \cites{CV,Knobloch-10} that
for each $v\in\Usm$, there exist positive constants
$\kappa_v$ and $\gamma_v$ such that
\begin{equation*}
\bnorm{\Prob^v_x(X_t\in \,\cdot\,\,|\,t<\uptau) -\alpha_v(\cdot)}_{\mathsf{TV}}
\,\le\, \kappa_v\,\E^{-\gamma_v t}\qquad\forall\,(t,x)\in\RR_+\times D\,,
\end{equation*}
with $\norm{\,\cdot\,}_{\mathsf{TV}}$ denoting the total variation norm.

%%%%%%%%%%%%%%%%%%%%%%%%%%%%%%%%%%%%%%%%%%%%%%%%%%%%%%%%%%%%%%%%%%%%%%%%%%%%%%%%%
%\begin{remark}
%Let $\Prob^{x,t}_v$ denote the probability
%measure on the canonical path space $\{X_s\colon 0\le s\le t\wedge\uptau\}$
%of the diffusion
%\cref{sde} under a control $v\in\Usm$, and
%$\widetilde\Prob^{x,t}_v$ the corresponding probability measure
%of the $Q$-process.
%By the Cameron--Martin--Girsanov theorem we obtain
%\begin{equation*}
%\frac{\D\mathbb{P}^{x,t}_v}{\D \widetilde\Prob^{x,t}_v} \,=\,
%\exp\biggl( -\int_0^{t} \bigl\langle \grad\psi_v(\widetilde{X}_s),
%\upsigma(\widetilde{X}_s) \D{\widetilde{W}_s}\bigr\rangle
%-\frac{1}{2}\int_0^{t}
%\babs{\upsigma\transp(\widetilde{X}_s)
%\grad \psi_v(\widetilde{X}_s)}^2\,\D{s}\biggr)\,.
%\end{equation*}
%Thus, the \emph{relative entropy}, or Kullback--Leibner
%divergence between $\widetilde\Prob^{x,t}_v$ and $\Prob^{x,t}_v$ takes the form
%\begin{equation*}
%D_{\mathsf{KL}}\bigl(\widetilde\Prob^{x,t}_v \bigm\| \Prob^{x,t}_v\bigr)
%\,=\, -\int \log\biggl(\frac{\D \mathbb{P}^{x,t}_v}{\D \widetilde\Prob^{x,t}_v}\biggr)\,
%\D \widetilde\Prob^{x,t}_v
%\,=\, \frac{1}{2}\,\widetilde\Exp^{x,t}_v\biggl[
%\int_0^{t} \babs{\upsigma\transp(\widetilde{X}_s)
%\grad \psi_v(\widetilde{X}_s)}^2\,\D{s}\biggr]\,.
%\end{equation*}
%Dividing this by $t$, and letting $t\searrow0$, we see that
%$\cH$ is the \emph{infinitesimal relative entropy rate}.
%\end{remark}

%%%%%%%%%%%%%%%%%%%%%%%%%%%%%%%%%%%%%%%%%%%%%%%%%%%%%%%%%%%%%%%%%%%%%%%%%%%%%%%%
\section{Variational formulation}\label{S4}

In this section we first derive a  variational formula for the controlled eigenvalue
that can be viewed as an abstract Collatz--Wielandt formula as in
\cites{VenBor-17,ABis-19,ABBK-19,ABK} and then map it to an expression in terms of the
$Q$-process in view of the foregoing developments.

For $v\in\Usm$, let
\begin{equation*}
\begin{aligned}
\Ag_{v,w}f(x) &\,\df\,
\Lg_v f(x) + \bigl\langle a(x)w, \nabla f(x)\bigr\rangle\,,
\quad\text{for\ } f \in C^2(D)\cap C(\overline{D})\,,\\
\cH_v &\,\df\, \biggl\{\nu \in \cP(D\times\Rd) \,\colon
\int_{D\times\Rd}\Ag_{v,w}f(x)\,\nu(\D{x},\D{w}) = 0
\quad\forall \, f \in C^2(D)\cap C(\overline{D})\biggr\}\,,
\end{aligned}
\end{equation*}
where $\cP(D\times\Rd)$ denotes the set of probability measures
on the Borel $\sigma$-algebra of $D\times\Rd$.
We refer to $\cH_v$ as the set of \emph{infinitesimal ergodic occupation
measures} of the operator $\Ag_{v,w}$.
A measure $\nu\in\cH_v$ can be disintegrated into
$\mu(\D{x}) \eta(\D{w}\,|\, x)$, with $\eta$ corresponding to a
(randomized) stationary Markov control.
We wish to emphasize that the diffusion with generator $\Ag_{v,w}$
under such a control $\eta$ is not, in general, confined to $D$.
Therefore, the elements of $\cH_v$ are not necessarily
ergodic occupation measures of a controlled diffusion process.
Nevertheless, the measure $\Tilde\mu_v(\D{x})\delta_{\grad\psi_v}(\D{w})$
lies in $\cH_v$ and is indeed an ergodic occupation measure of
the controlled diffusion with generator $\Ag_{v,w}$.

We also define
\begin{equation}\label{Enufin}
\begin{aligned}
\cPsv &\,\df\,
\biggl\{\nu\in  \cP(D\times\Rd)\,\colon
\int_{D\times\Rd} \frac{\babs{\upsigma\transp(x)\grad\psi_v}^2}{1+\abs{\psi_v}}\,
\nu(\D{x},\D{w}) <\infty\biggr\}\,,\\
\cPov &\,\df\,
\biggl\{\nu\in  \cP(D\times\Rd)\,\colon
\int_{D\times\Rd} \babs{\upsigma\transp(x)w}^2\,\nu(\D{x},\D{w}) <\infty\biggr\}\,.
\end{aligned}
\end{equation}

In the proof of the results which follow, we make use of a family
of cutoff functions defined as follows.

\begin{definition}
For $t>0$, we let $\chi^{}_t$ be a convex $\Cc^2(\RR)$ function such that
$\chi^{}_t(s)= s$ for $s\ge -t$, and $\chi^{}_t(s) =\text{constant}$ for $s\le -t\E^2$.
Then $\chi'_t$ and $\chi''_t$ are nonnegative.
In addition, we select $\chi^{}_t$ so that
\begin{equation*}
\chi''_t(s) \,\le\, -\frac{1}{s}\qquad
\text{for\ } s\in[-t\E^2,-t]\ \ \text{and\ } t\ge0\,.
\end{equation*}
This is always possible.
\end{definition}

%%%%%%%%%%%%%%%%%%%%%%%%%%%%%%%%%%%%%%%%%%%%%%%%%%%%%%%%%%%%%%%%%%%%%%%%%%%%%%%%
\begin{lemma}\label{L4.1}
For any $v\in\Usm$ we have
\begin{equation*}
\lambda_v \,=\, \inf_{\nu\in\cH_v\cap\cPsv}\,\frac{1}{2} \int_{D\times\Rd}
\babs{\upsigma\transp(x)w}^2\,\nu(\D{x},\D{w})\,.
\end{equation*}
\end{lemma}

\begin{proof}
We let
\begin{equation*}
L(x,w) \,\df\, \frac{1}{2} \abs{\upsigma\transp(x)w}^2\,.
\end{equation*}
Using the definition in \cref{E-TLgv}, we have the identity
\begin{equation}\label{PL4.1A}
\begin{aligned}
\Ag_{v,w}\psi_v(x)
&\,=\, \Lg_v \psi_v(x) +
\frac{1}{2}\babs{\upsigma\transp(x)\grad\psi_v(x)}^2+ L(x,w)
-\frac{1}{2}\babs{\upsigma\transp(x)\bigl(w-\grad\psi_v(x)\bigr)}^2\\
&\,=\,L(x,w)- \lambda_v
-\frac{1}{2}\babs{\upsigma\transp(x)\bigl(w-\grad\psi_v(x)\bigr)}^2\,,
\end{aligned}
\end{equation}
where the second equality follows from \cref{E-eigenA}.
Since
\begin{equation*}
\Ag_{v,w} \chi^{}_t(\psi_v) \,=\, \chi'_t(\psi_v) \Ag_{v,w}\psi_v
+ \tfrac{1}{2} \chi''_t(\psi_v)
\babs{\upsigma\transp\grad\psi_v}^2\,,
\end{equation*}
we obtain from \cref{PL4.1A} that
\begin{equation}\label{PL4.1B}
\begin{aligned}
\Ag_{v,w} \chi^{}_t\bigl(\psi_v(x)\bigr)&
- \chi''_t\bigl(\psi_v(x)\bigr)\,L\bigl(x,\grad\psi_v(x)\bigr)\\
& = \chi'_t\bigl(\psi_v(x)\bigr)\Bigl(
L(x,w) - \tfrac{1}{2} \babs{\upsigma\transp(x)\bigl(w-\grad\psi_v(x)\bigr)}^2
- \lambda_v\Bigr)\,.
\end{aligned}
\end{equation}
Let $\nu\in\cH_v\cap\cPsv$, and without loss of generality assume
$\nu\in\cPov$.
Then $\int \Ag_{v,w} \chi^{}_t\bigl(\psi_v(x)\bigr)\,\D\nu=0$ by
the definition of $\cH_v$,
and
\begin{equation*}
\int_{D\times\Rd}
\chi''_t\bigl(\psi_v(x)\bigr)\,\,L\bigl(x,\grad\psi_v(x)\bigr)\,\nu(\D{x},\D{w})
\,\xrightarrow[t\to\infty]{}\,0
\end{equation*}
by the definitions of $\chi^{}_t$ and $\cPsv$.
Thus, integrating \cref{PL4.1B} with
respect to $\nu$ and letting $t\to\infty$, using also monotone convergence,
we obtain
\begin{equation}\label{PL4.1C}
\lambda_v + \frac{1}{2} \int_{D\times\Rd}
\babs{\upsigma\transp(x)\bigl(w-\grad\psi_v(x)\bigr)}^2\,\nu(\D{x},\D{w})
\,=\,\int_{D\times\Rd} L(x,w)\,\nu(\D{x},\D{w})\,.
\end{equation}
As mentioned earlier, $\Tilde\mu_v(\D{x})\delta_{\grad\psi_v}(\D{w})\in\cH_v$,
and also lies in $\cPsv$ by \cref{ET3.1B}.
Therefore, the result follows by \cref{ET3.1B,PL4.1C}.
\end{proof}

We continue with a variational formula for $\lamstr$.
Let
\begin{equation*}
\Ag f(x,u,w) \,\df\, \Lg f(x,u) +\bigl\langle a(x)w, \nabla f(x)\bigr\rangle\,,
\quad (x,u,w)\in D\times\Act\times\Rd\,.
\end{equation*}
As in \cref{PL4.1A}, we have
\begin{equation}\label{EvarA}
\Ag f(x,u,w)\,=\, \Lg f(x,u) + \frac{1}{2}\babs{\upsigma\transp(x)\grad f(x)}^2+ L(x,w)
-\frac{1}{2}\babs{\upsigma\transp(x)\bigl(w-\grad f(x)\bigr)}^2
\end{equation}
for all $f\in\Cc^2(D)$,
and thus we obtain
\begin{equation*}
\begin{aligned}
\Ag \psi^*(x,u,w)- L(x,w) &\,=\,
\Lg \psi^*(x,u) +
\frac{1}{2}\babs{\upsigma\transp(x)\grad \psi^*(x)}^2
-\frac{1}{2}\babs{\upsigma\transp(x)\bigl(w-\grad \psi^*(x)\bigr)}^2\\
&\,\le\, -\lamstr -\frac{1}{2}\babs{\upsigma\transp(x)\bigl(w-\grad \psi^*(x)\bigr)}^2
\quad \forall\,(x,u,w)\in D\times\Act\times\Rd\,.
\end{aligned}
\end{equation*}

Recall from \cref{Snot} that $\Cc^2_+(D)$ denotes the set of functions
in $\Cc^2(D)$ which are positive on $D$.
The starting point of the analysis
 is \cite[Theorems~2.1 and 2.5]{ABis-19} which assert that
\begin{equation}\label{E-CW}
\begin{aligned}
-\lamstr &\,=\,
\adjustlimits\inf_{h\in\Cc^2_+(D)\cap\Cc(\overline{D})\;}\sup_{\mu\in\cP(D)}\;
\int_{D}\frac{\cG^* h(x)}{h(x)}\, \mu(\D{x})\\
&\,=\, \adjustlimits\sup_{\mu\in\cP(D)\;} \inf_{h\in\Cc^2_+(D)\cap\Cc(\overline{D})}\;
\int_{D}\frac{\cG^* h(x)}{h(x)}\, \mu(\D{x})\,,
\end{aligned}
\end{equation}
with $\cG^*$ as in \cref{E-cGmax}.
We claim that we may replace $\Cc^2_+(D)\cap\Cc(\overline{D})$ with
$\Cc^2_+(\overline{D})$ in the first equality of \cref{E-CW}.
To prove the claim first
note that \cref{E-CW} implies that
\begin{equation}\label{EvarC}
-\lamstr \,\le\, \adjustlimits\inf_{h\in\Cc^2_+(\overline{D})\;} \sup_{\mu\in\cP(D)}\;
\int_{D}\frac{\cG^* h(x)}{h(x)}\, \mu(\D{x})\,.
\end{equation}
Now let $\epsilon>0$ be arbitrary
and $D'\Supset D$ be a bounded $\Cc^{2,1}$ domain such that
$\lambda_{D'}(\cG^*) \ge \lamstr -\epsilon$.
Let $\Phi^*$ denote the principal eigenfunction of $\cG^*$ on $D'$.
Then $\Phi^*\in\Cc^2_+(\overline{D})$ and
$-\lamstr+\epsilon\ge\frac{\cG^*\Phi^*}{\Phi^*}$ on $\overline{D}$, which implies
equality in \cref{EvarC}.
Note that the infimum in \cref{EvarC} is not attained in
$\Cc^2_+(\overline{D})$, but working in this space allows us
to obtain a representation formula which does not rely on $\cPsv$
in \cref{Enufin} as is the case in \cref{L4.1}.

We define
\begin{equation*}
F(f,\uppi) \,\df\,
\int_{D\times\Act\times\Rd}
\bigl(\Ag f(x,u,w)-L(x,w)\bigr)\,\uppi(\D{x},\D{u},\D{w})
\end{equation*}
for $f\in\Cc^{2}(\overline{D})$
and $\uppi\in\cP(D\times\Act\times\Rd)$.
Consider $h\in\Cc^2_+(\overline{D})$, and let $f=\log h$.
Then $f\in\Cc^2(\overline{D})$, and using \cref{EvarA}, a simple calculation shows that
\begin{equation*}
\frac{\cG^*h(x)}{h(x)}\,=\, \sup_{(u,w)\in\Act\times\Rd}\,
\bigl[\Ag f(x,u,w) - L(x,w)\bigr]\,.
\end{equation*}
Combining this with \cref{EvarC}, for which we have already shown
that equality holds, we obtain
\begin{equation}\label{EvarD}
-\lamstr\,=\, \adjustlimits
\inf_{f\in\Cc^2(\overline{D})\;} \sup_{\uppi\in\cP(D\times\Act\times\Rd)}\;
F(f,\uppi)\,.
\end{equation}

Define
\begin{equation*}%\label{E-MA}
\cM_\Ag \,\df\, \biggl\{\uppi \in \cP(D\times\Act\times\Rd)\, \colon
\int_{D\times\Act\times\Rd}\Ag f\,\D\uppi = 0
\quad\forall \, f \in C^2(\overline{D})\biggr\}\,.
\end{equation*}
We have the following result.

%%%%%%%%%%%%%%%%%%%%%%%%%%%%%%%%%%%%%%%%%%%%%%%%%%%%%%%%%%%%%%%%%%%%%%%%%%%%%%%%
\begin{lemma}\label{L4.2}
It holds that
\begin{equation}\label{EL4.2A}
-\lamstr \,=\, \adjustlimits
 \sup_{\uppi\in\cP(D\times\Act\times\Rd)\;} \inf_{f\in\Cc^2(\overline{D})}\;
F(f,\uppi) \,=\,
-\inf_{\uppi\in\cM_\Ag}\,
\int_{D\times\Act\times\Rd} L(x,w)\,\uppi(\D{x},\D{u},\D{w})\,.
\end{equation}
\end{lemma}

\begin{proof}
Let
\begin{equation}\label{PL4.2A}
\rho \,\df\, \adjustlimits\sup_{\uppi\in\cP(D\times\Act\times\Rd)\;}
\inf_{f\in\Cc^2(\overline{D})}\; F(f,\uppi)\,.
\end{equation}
It follows by \cref{EvarD,PL4.2A} that $\rho\le - \lamstr$.
It is also clear that if $\uppi\notin\cM_\Ag$ then
$\inf_{f\in\Cc^2(\overline{D})}\, F(f,\uppi)=-\infty$, so we assume
that $\uppi\in\cM_\Ag$.
Since the second and third terms in \cref{EL4.2A} are equal
when $\uppi\in\cM_\Ag$, this also shows that
\begin{equation}\label{PL4.2B}
\inf_{\uppi\in\cM_\Ag}\,
\int_{D\times\Act\times\Rd} L(x,w)\,\uppi(\D{x},\D{u},\D{w})\,\ge\,
\lamstr\,.
\end{equation}
It remains to show equality in \cref{PL4.2B}.
Recall that $\Usms$ denotes the set of measurable selectors from the minimizer
in \cref{E-Dir*}.
For $v^*\in\Usms$, the $Q$-process is ergodic by
\cref{T3.1}, and thus its invariant measure $\Tilde\mu_{v^*}$
satisfies
\begin{equation*}
\int_{D} \widetilde\Lg_{v^*} f(x)\,\Tilde\mu_{v^*}(\D{x})\,=\,0
\qquad\forall f\in\Cc^2(\overline{D})\,.
\end{equation*}
It is also clear that the measure $\uppi^*$, defined by
\begin{equation}\label{E-pi*}
\uppi^*(\D{x},\D{u},\D{w})
\,\df\,\Tilde\mu_{v^*}(\D{x})\,v^*(\D{u}\,|\,x)\,\delta_{\grad\psi^*(x)}(\D{w})\,,
\end{equation}
satisfies
\begin{equation*}
\int_{D\times\Act\times\Rd}
\Ag f(x,u,w)\,\uppi^*(\D{x},\D{u},\D{w})\,=\,
\int_{D} \widetilde\Lg_{v^*} f(x)\,\Tilde\mu_{v^*}(\D{x})\,=\,0
\qquad\forall f\in\Cc^2(\overline{D})\,,
\end{equation*}
which implies that $\uppi^*\in\cM_\Ag$.
On the other hand
\begin{equation*}
\lamstr\,=\,\int_{D} L\bigl(x,\grad\psi^*(x)\bigr)\,
\,\Tilde\mu_{v^*}(\D{x})\,=\,
\int_{D\times\Act\times\Rd}
L(x,w)\,\uppi^*(\D{x},\D{u},\D{w})
\end{equation*}
by \cref{ET3.1B}\,,
which shows that the infimum $\inf_{\uppi\in\cM_\Ag}\int L\,\D\uppi$
is attained at $\uppi^*$.
This completes the proof.
\end{proof}

We gather the results of \cref{EvarD,L4.2} in the following theorem,
which also characterizes the measures $\uppi\in\cM_\Ag$ which attain
the infimum in \cref{EL4.2A}.

%%%%%%%%%%%%%%%%%%%%%%%%%%%%%%%%%%%%%%%%%%%%%%%%%%%%%%%%%%%%%%%%%%%%%%%%%%%%%%%%
\begin{theorem}\label{T4.1}
We have
\begin{equation*}
-\lamstr \,=\, \adjustlimits\inf_{f \in \Cc^2_+(\overline{D})\,}
\sup_{\uppi\in\cP(D\times\Act\times\Rd)}\,F(f,\uppi)
\,=\,
\adjustlimits\sup_{\uppi\in\cP(D\times\Act\times\Rd)\,}
\inf_{f \in \Cc^2_+(\overline{D})}\,F(f,\uppi)\,,
\end{equation*}
\and
\begin{equation}\label{ET4.1A}
\lamstr \,=\, \min_{\uppi\in\cM_\Ag}\,
\frac{1}{2}\,\int_{D\times\Act\times\Rd} \abs{\upsigma\transp(x)w}^2\,
\uppi(\D{x},\D{u},\D{w})\,.
\end{equation}
In addition, any $\uppi\in\cM_\Ag$ which attains the
minimum in \cref{ET4.1A}
has the form in \cref{E-pi*} for some $v^*\in\Usms$.
\end{theorem}

\begin{proof}
We need to prove the assertion in the second part of the theorem.
Let $D_n\Supset D$, $n\in\NN$, be a decreasing sequence of bounded $C^{2,1}$ domains
such that $\cap_{n\in\NN} D_n = \overline{D}$.
We denote the principal eigenpair of $\cG^*$ on $D_n$ by
$(\lambda^*_n,\Phi_n)\in\RR\times\Cc^2(\overline{D}_n)$, and set $\phi_n=\log\Phi_n$.
Suppose $\widehat\uppi\in\cM_\Ag$ attains the minimum in \cref{ET4.1A}.
We disintegrate it as
\begin{equation*}
\widehat\uppi(\D{x},\D{u},\D{w}) \,=\,\Hat\mu(\D{x})\,
\Hat{v}(\D{u}\,|\,x)\, \widehat{w}(\D{w}\,|\, u,x)\,.
\end{equation*}
Using \cref{EvarA}, with $f=\phi_n\in\Cc^2_+(\overline{D})$, and integrating
with respect to $\widehat\uppi$ we obtain
\begin{equation}\label{PT4.1A}
\begin{aligned}
-\lamstr&\,=\, - \int_{D\times\Act\times\Rd} L(x,w)\,
\widehat\uppi(\D{x},\D{u},\D{w})\\
&\,=\, \int_{D}
\biggl(\Lg_{\Hat{v}} \phi_n(x) +
\frac{1}{2}\babs{\upsigma\transp(x)\grad \phi_n(x)}^2\biggr)\,\Hat\mu(\D{x})\\
&\mspace{50mu}-
\int_{D\times\Act\times\Rd}
\frac{1}{2}\babs{\upsigma\transp(x)\bigl(w-\grad \phi_n(x)\bigr)}^2\,
\widehat{w}(\D{w}\,|\, u,x)\,\Hat\eta(\D{x},\D{u})\,,
\end{aligned}
\end{equation}
with $\Hat\eta(\D{x},\D{u})=\Hat\mu(\D{x})\,\Hat{v}(\D{u}\,|\,x)$.
Now
\begin{equation}\label{PT4.1B}
\Lg_{\Hat{v}} \phi_n(x) +
\frac{1}{2}\babs{\upsigma\transp(x)\grad \phi_n(x)}^2 \,\le\, -\lambda^*_n\,,
\end{equation}
and we know that $\lambda^*_n\nearrow \lamstr$ as $n\to\infty$.
By \cref{PT4.1A}, we must have
\begin{equation}\label{PT4.1C}
\limsup_{n\to\infty}\,\int_{D}
\biggl(\Lg_{\Hat{v}} \phi_n(x) +
\frac{1}{2}\babs{\upsigma\transp(x)\grad \phi_n(x)}^2\biggr)\,\Hat\mu(\D{x})
\,\ge\, -\lamstr\,.
\end{equation}
By elliptic regularity, $\phi_n$
converges to $\psi^*$ in $\Cc^{2,\alpha}(K)$ for any
compact $K\subset D$ and any $\alpha\in(0,1)$.
Thus $\Lg_{\Hat{v}} \phi_n\to \Lg_{\Hat{v}} \psi^*$ pointwise
in $D$,
and in view of \cref{PT4.1B},
we can apply Fatou's lemma to \cref{PT4.1C} to obtain
\begin{equation*}
\int_{D}
\biggl(\Lg_{\Hat{v}} \psi^*(x) +
\frac{1}{2}\babs{\upsigma\transp(x)\grad \psi^*(x)}^2\biggr)\,\Hat\mu(\D{x})
\,=\, -\lamstr\,.
\end{equation*}
This shows that $\Hat{v}= v^*$ a.e.\ on the support of $\Hat\mu$
for some $v^*\in\Usms$.
Similarly, from the last term in \cref{PT4.1A}, we obtain
\begin{equation*}
\int_{D\times\Act\times\Rd}
\frac{1}{2}\babs{\upsigma\transp(x)\bigl(w-\grad \psi^*(x)\bigr)}^2\,
\widehat{w}(\D{w}\,|\, u,x)\,\Hat\eta(\D{x},\D{u})\,=\,0\,,
\end{equation*}
which shows that $\widehat{w}(\D{w}\,|\, u,x)=\delta_{\grad\psi^*(x)}(\D{w})$ a.e.\
on the support of $\Hat\eta(\D{x},\D{u})=\Hat\mu(\D{x})v^*(\D{u}\,|\,x)$.

Let
\begin{equation*}
\widetilde\uppi(\D{x},\D{u},\D{w}) \,=\, \Hat\mu(\D{x})\,
v^*(\D{u}\,|\,x)\,\delta_{\grad\psi^*(x)}(\D{w})\,.
\end{equation*}
Since $\widetilde\uppi$ agrees with $\widehat\uppi$ on the support
of $\Hat\mu$, we must have
$\int_{D\times\Act\times\Rd}\Ag f\,\D\widetilde\uppi = 0$
for all $f \in C^2(\overline{D})$,
which implies that
\begin{equation*}%\label{PT4.1D}
\int_{D} \widetilde\Lg_{v^*} f(x)\,
\Hat\mu(\D{x})\,=\,0 \qquad \forall\,f \in C^2(\overline{D})\,.
\end{equation*}
The Theorem in \cite{Echeverria-82}
(see also \cite[Chapter~8]{Ethier-Kurtz}) then asserts that
$\Hat\mu$ is an invariant probability measure for the process controlled
by $v^*$.
The uniqueness of the invariant probability measure of the $Q$-process
then implies that $\Hat\mu=\Tilde\mu_{v^*}$, which combined with the
argument in the preceding paragraph shows that
$\Hat{v}= v^*$ and $\widehat{w}=\delta_{\grad\psi^*(x)}(\D{w})$
a.e.\ $x\in D$.
This completes the proof.
\end{proof}

One may view \cref{ET4.1A}
as the abstract linear programming formulation of the ergodic control problem
for the controlled diffusion $\process{Z}$ in $D$ whose (controlled) extended
generator is given by $\Ag$ indexed by the control variables $u\in\Act$ and
$w\in\Rd$,
with the objective of minimizing the ergodic cost
\begin{equation*}%\label{ergodic-cost}
\limsup_{T\nearrow\infty}\,\frac{1}{T}\,
\Exp\biggl[\frac{1}{2}\int_0^T\babs{\upsigma\transp(Z_t)W_t}^2\,\D{t}\biggr]
\end{equation*}
over all non-anticipative control processes $\bigl(U_t,W_t\bigr)_{t\ge0}$
such that $\Prob(\uptau_\epsilon<t)\to0$ as $\epsilon\searrow0$
for any $t>0$,
with $\uptau_\epsilon$ as in \cref{R3.2}.
The corresponding Hamilton--Jacobi--Bellman equation is
\begin{equation*}
\min_{(u,w)\in\Act\times\Rd}\, \biggl(\Ag\Phi(x,u,w)
+ \frac{1}{2}\babs{\upsigma\transp(x)w}^2\biggr)
- \beta \,=\, 0\,,
\end{equation*}
where $\Phi\in\Cc^2(D)\cap\{f : \lim_{x \to \partial D}f(x) = -\infty\}$.
Performing the minimization over $w$, we obtain
\begin{equation*}
\min_{u\in\Act}\,\Lg\Phi(x,u)
- \frac{1}{2}\babs{\upsigma\transp(x)\nabla\Phi(x)}^2 - \beta \,=\, 0\,.
\end{equation*}
Comparing with \cref{E-eigenB}, we have $\Phi = -\psi^*$ and $\beta = \lamstr$.
Here we use the well-posedness of  \cref{E-eigenB} which follows from the fact
that the invertible smooth transformation $\Psi^* = e^{\psi^*}$ converts
\cref{E-eigenB}  into  the well posed Dirichlet problem \cref{E-Dir*} with
a unique solution
in $\Cc^2(D)\cap \Cc(\bar{D})$.
Recall that $\Usms$ denotes the set of measurable  selectors from the minimizer
in \cref{E-Dir*}, which is precisely the set of optimal stationary Markov controls.
That is, $\lamstr=\lambda_{v}=\beta_v$ and
$\Psi_v=\Psi^*$ for all $v\in\Usms$.
Therefore, under any optimal choice $v^*\in\Usms$, the optimal choice of $\process{W}$
is precisely $W_t = \nabla\psi^*(Z_t)$, $t \ge 0$.
That is, the optimal process is the corresponding $Q$-process $\process{\widetilde{Z}}$
controlled by $v^*$.
This equivalence leads to the following representation for the optimal eigenvalue
(i.e., optimal exit rate) in terms of the $Q$-process.

%%%%%%%%%%%%%%%%%%%%%%%%%%%%%%%%%%%%%%%%%%%%%%%%%%%%%%%%%%%%%%%%%%%%%%%%%%%%%%%%
\begin{theorem}
For any $v^*\in\Usms$, the optimal exit rate $\lamstr$ is given by
\begin{align*}
\lamstr &\,=\, \frac{1}{2}
\int_D\babs{\upsigma\transp(x)\nabla\psi^*(x)}^2\,\Tilde\mu_{v^*}(\D{x}) \\[5pt]
&\,=\, \inf_{v\in\Usm}\,\frac{1}{2}
\int_D\babs{\upsigma\transp(x)\nabla\psi_v(x)}^2\,\Tilde\mu_{v}(\D{x}) \\[5pt]
&\,=\,\frac{\int_D\frac{\abs{\upsigma\transp(x)\nabla\Psi^*(x)}^2}{\Psi^*(x)}\,
\alpha_{v^*}(\D{x})}{2\int_D\Psi^*(x)\,\alpha_{v^*}(\D{x})} \\[5pt]
&\,=\, \inf_{v \in \Usm}\,
\frac{\int_D\frac{\abs{\upsigma\transp(x)\nabla\Psi_v(x)}^2}{\Psi_v(x)}\,
\alpha_{v}(\D{x})}{2\int_D\Psi_v(x)\,\alpha_{v}(\D{x})}\,.
\end{align*}
\end{theorem}

%%%%%%%%%%%%%%%%%%%%%%%%%%%%%%%%%%%%%%%%%%%%%%%%%%%%%%%%%%%%%%%%%%%%%%%%%%%%%%%%

%%%%%%%%%%%%%%%%%%%%%%%%%%%%%%%%%%%%%%%%%%%%%%%%%%%%%%%%%%%%%%%%%%%%%%%%%%%%%%%%
\section*{Acknowledgements}
The work of Ari Arapostathis was supported in part by
the Army Research Office through grant W911NF-17-1-001,
in part by the National Science Foundation through grant DMS-1715210,
and in part by the Office of Naval Research through grant N00014-16-1-2956
and was approved for public release under DCN\# 43-6053-19.
The work of Vivek Borkar was supported by a J.\ C.\ Bose Fellowship.

%%%%%%%%%%%%%%%%%%%%%%%%%%%%%%%%%%%%%%%%%%%%%%%%%%%%%%%%%%%%%%%%%%%%%%%%%%%%%%%%
%\bibliography{KIEV.bib}
%\end{document}
%%%%%%%%%%%%%%%%%%%%%%%%%%%%%%%%%%%%%%%%%%%%%%%%%%%%%%%%%%%%%%%%%%%%%%%%%%%%%%%%

%%%%%%%%%%%%%%%%%%%%%%%%%%%%%%%%%%%%%%%%%%%%%%%%%%%%%%%%%%%%%%%%%%%%%%%%%%%%%%%%
\end{document}